\documentclass[dvipdfmx,12pt]{amsart}
\usepackage{amsfonts}
\usepackage{amssymb}
\usepackage{amscd}
\usepackage{mathrsfs}
\usepackage{tikz}

\newtheorem{theorem}{Theorem}[section]
\newtheorem{corollary}[theorem]{Corollary}

\newtheorem{lemma}[theorem]{Lemma}
\newtheorem{proposition}[theorem]{Proposition}

\newtheorem*{theorem*}{Theorem}
\newtheorem*{corollary*}{Corollary}
\newtheorem*{conjecture*}{Conjecture}
\newtheorem*{lemma*}{Lemma}
\newtheorem*{proposition*}{Proposition}
\newtheorem*{problem*}{Problem}
\newtheorem*{axiom*}{Axiom}
\newtheorem*{example*}{Example}
\newtheorem*{exercise*}{Exercise}
\newtheorem*{definition*}{Definition}
\theoremstyle{definition}
\newtheorem{remark}{Remark}[section]
\newtheorem*{remark*}{Remark}
\numberwithin{equation}{section}

\newcommand{\im}{{\rm Im}}
\newcommand{\re}{{\rm Re}}
\newcommand{\supp}{{\rm supp }}

\renewcommand{\bigskip}{\vspace{0.2cm}}
\renewcommand{\l}{\left}
\renewcommand{\r}{\right}

\newcommand{\cleq}{\lesssim}

\newcommand{\norm}{\Vert}

\newcommand{\wto}{\rightharpoonup}
\newcommand{\til}{\widetilde}

\def\eqref[#1]{\textup{(\ref{#1})}}
\def\secref[#1]{Section~\ref{#1}}
\def\remref[#1]{Remark~\ref{#1}}
\def\appref[#1]{Appendix~\ref{#1}}
\def\lemref[#1]{Lemma~\ref{#1}}
\def\corref[#1]{Corollary~\ref{#1}}
\def\thmref[#1]{Theorem~\ref{#1}}
\def\propref[#1]{Proposition~\ref{#1}}
\def\conjref[#1]{Conjecture~\ref{#1}}
\def\defref[#1]{Definition~\ref{#1}}
\def\abs[#1]{|#1|}
\def\norm[#1]{\left\Vert #1 \right\Vert}
\def\jbra[#1]{\langle #1 \rangle}
\def\tbra[#1,#2]{\left\langle #1 , #2\right\rangle} 
\def\rbra[#1,#2]{\left( #1 , #2 \right)} 
\def\sbra[#1,#2]{\left[ #1 , #2 \right]} 
\def\fbra[#1,#2]{\{ #1 | #2 \}} 
\def\besov[#1,#2,#3]{B_{#2,#3}^{#1}}
\def\hbesov[#1,#2,#3]{\dot{B}_{#2,#3}^{#1}}

\newcommand{\N}{{\mathbb N}}
\newcommand{\R}{{\mathbb R}}

\newcommand{\cE}{{\mathcal E}}

\newcommand{\cK}{{\mathcal K}}

\newcommand{\cP}{{\mathcal P}}

\newcommand{\cS}{{\mathcal S}}

\newcommand{\scA}{{\mathscr A}}

\newcommand{\scG}{{\mathscr G}}

\newcommand{\scK}{{\mathscr K}}

\newcommand{\scM}{{\mathscr M}}

\begin{document}
\title[G.E. for gDNLS]{A sufficient condition for global existence of solutions to a generalized derivative nonlinear Schr\"{o}dinger equation}
\author[N. Fukaya]{Noriyoshi Fukaya}
\address{Department of Mathematics, Graduate School of Science, Tokyo University of Science, Shinjuku, Tokyo, 162-8601, Japan}
\email{1116702@ed.tus.ac.jp}
\author[M. Hayashi]{Masayuki Hayashi}
\address{Department of Applied Physics, Waseda University,  Shinjuku, Tokyo 169-8555, Japan}

\email{masayuki-884@fuji.waseda.jp}
\author[T. Inui]{Takahisa Inui}
\address{Department of Mathematics, Graduate School of Science, Kyoto University, Kyoto City, Kyoto, 606-8502, Japan}
\email{inui@math.kyoto-u.ac.jp}
\date{}
\maketitle

\begin{abstract}
We give a sufficient condition for global existence of the solutions to a generalized derivative nonlinear Schr\"{o}dinger equation (gDNLS) by a variational argument. 
The variational argument is applicable to a cubic derivative nonlinear Schr\"{o}dinger equation (DNLS). For (DNLS), Wu \cite{Wu15} proved that the solution with the initial data $u_0$ is global if $\left\Vert u_0 \right\Vert_{L^2}^2<4\pi$ by the sharp Gagliardo--Nirenberg inequality. The variational argument gives us another proof of the global existence for (DNLS). Moreover, by the variational argument, we can show that the solution to (DNLS) is global if the initial data $u_0$ satisfies that $\left\Vert u_0 \right\Vert_{L^2}^2=4\pi$ and the momentum $P(u_0)$ is negative. 
\end{abstract}

\tableofcontents


\section{Introduction}

\subsection{Background}

The following equation is known as a derivative nonlinear Schr\"{o}dinger equation:
\begin{equation}
\label{eq1.1}
i \partial_t v + \partial_x^2 v +i  \partial_x(|v|^{2} v )=0 , \quad (t,x) \in \R \times \R.
\end{equation}
This equation appears in plasma physics (see \cite{MOMT76, M76}) and as a model for ultrashort optical pulses (see \cite{MMW07}). 
Using the gauge transformation
\[ u(t,x) = v(t,x)\exp\l( \frac{i}{2} \int_{-\infty}^{x} |v(t,x)|^2 dx\r), \]
we get a Hamiltonian form of the equation \eqref[eq1.1]: 
\begin{equation}
\label{DNLS}
\tag{DNLS}  
i \partial_t u + \partial_x^2 u +i |u|^{2} \partial_x u =0 , \quad (t,x) \in \R \times \R.
\end{equation}
Namely, this equation can be written as $i \partial_t u = E ' (u)$ (see below for the definition of the Hamiltonian $E$). The Cauchy problem for \eqref[DNLS] (or equivalently \eqref[eq1.1]) has been studied by many researchers. It is known that \eqref[DNLS] is locally well-posed in the energy space $H^1(\R)$ (see \cite{TF80, HO92, H93, HO94_1, HO94_2}). Hayashi and Ozawa proved that the solution is global if $\Vert u_0 \Vert_{L^2}^2 < 2\pi$ (see \cite{HO94_1, O96}). Wu proved that it holds if $\Vert u_0 \Vert_{L^2}^2 < 4\pi$ (see \cite{Wu13, Wu15}). Recently, Miao, Tang, and Xu obtained the global well-posedness by a variational argument (see Remark \ref{rem1.5} below). 
For the initial data with low regularity, there are also many works. 
Takaoka \cite{T99} proved that \eqref[DNLS] is locally well-posed in $H^s(\R)$ when $s\geq 1/2$ by the Fourier restricted method. Biagioni and Linares \cite{BL01} proved that the solution map from $H^s(\R)$ to $C([-T,T]:H^s(\R))$, where $T>0$,  for \eqref[DNLS] is not locally uniformly continuous when $s <1/2$. Colliander, Keel, Staffilani, Takaoka, and Tao \cite{CKSTT02} proved that the $H^s$-solution is global if $\Vert u_0 \Vert_{L^2}^2 < 2\pi$ when $s>1/2$ by the $I$-method (see also \cite{CKSTT01, T01}).  Recently, Miao, Wu, and Xu \cite{MWX11} showed that $H^{1/2}$-solution is global if $\Vert u_0 \Vert_{L^2}^2 < 2\pi$. Guo and Wu \cite{GW17} improved their result, that is, they proved that $H^{1/2}$-solution is global if $\Vert u_0 \Vert_{L^2}^2 < 4\pi$.
The orbital stability of solitary waves has been also studied. It is known that \eqref[DNLS] has a two-parameter family of the solitary waves
$u_{\omega,c}(t,x)=e^{i\omega t} \phi_{\omega,c}(x-ct)$,
where $(\omega,c)$ satisfies $\omega > c^2/4$, or $\omega = c^2/4$ and $c>0$ (see below for the explicit formula of $\phi_{\omega,c}$). Boling Guo and Yaping Wu \cite{GW95} proved that the solitary waves $u_{\omega,c}$ are orbitally stable when $\omega>c^2/4$ and $c<0$ by the abstract theory of Grillakis, Shatah, and Strauss \cite{GSS87, GSS90} and the spectral analysis of the linearized operators. Colin and Ohta \cite{CO06} proved that the solitary waves $u_{\omega,c}$ are orbitally stable when $\omega>c^2/4$ by characterizing the solitary waves from the view point of a variational structure. The case of $\omega=c^2/4$ and $c>0$ is treated by Kwon and Wu \cite{KW16pre}. Recently, the stability of the multi-solitons is studied by Miao, Tang, and Xu \cite{MTX16pre_2} and Le Coz and Wu \cite{LCW16pre}. 

To understand the structural properties of \eqref[DNLS], Liu, Simpson, and Sulem introduced an extension of \eqref[DNLS] with general power nonlinearity (see \cite{LSS13}). 
The generalized derivative nonlinear Schr\"{o}dinger equation is
\begin{equation}
\tag{gDNLS}
\label{gDNLS}
\l\{
\begin{array}{ll}
i \partial_t u + \partial_x^2 u +i |u|^{2\sigma} \partial_x u =0 , & (t,x) \in \R \times \R,
\\
u(0,x)=u_0(x), & x\in \R,
\end{array}
\r.
\end{equation}
where $\sigma >0$. The equation (gDNLS) is invariant under the scaling transformation 
\begin{align*}
u_{\gamma}(t,x):=\gamma^{\frac{1}{2\sigma}} u(\gamma^2 t,\gamma x), \quad \gamma>0. 
\end{align*}
This implies that its critical Sobolev exponent is $ s_c=\frac{1}{2}-\frac{1}{2\sigma}$. In particular, the equation (DNLS) is $L^2$-critical. 
In \cite{LSS13}, Liu, Simpson, and Sulem investigated the orbital stability of a two-parameter family of solitary waves
\[ u_{\omega,c}(t,x)=e^{i\omega t} \phi_{\omega,c}(x-ct), \]
where $(\omega,c)$ satisfies $\omega > c^2/4$, or $\omega = c^2/4$ and $c>0$,
\begin{align} 
\label{eq1.2}
\phi_{\omega,c}(x) 
&= \Phi_{\omega,c}(x) \exp\l( i\frac{c}{2}x - \frac{i}{2\sigma +2} \int_{0}^{x} \Phi_{\omega,c}(y)^{2\sigma} dy \r),
\\
\label{eq1.3}
\Phi_{\omega,c}(x) &=
\l\{
\begin{array}{ll}\displaystyle
\l\{  \frac{(\sigma +1)(4\omega - c^2)}{2\sqrt{\omega} \cosh (\sigma \sqrt{4\omega- c^2}x)-c} \r\}^{\frac{1}{2\sigma}}, 
& \displaystyle \text{if } \omega > \frac{c^2}{4},
\\
 & 
\\ \displaystyle
\l\{\frac{2(\sigma+1) c}{\sigma^2(cx)^2+1}\r\}^{\frac{1}{2\sigma}}, & \displaystyle \text{if } \omega =\frac{c^2}{4} \text{ and } c>0.
\end{array}
\r.
\end{align}
We note that $\Phi_{\omega,c}$ is the positive even solution of 
\begin{align}
\label{elliptic}
- \Phi ''+ \l(\omega- \frac{c^2}{4}\r) \Phi +\frac{c}{2} |\Phi|^{2\sigma} \Phi - \frac{2\sigma+1}{(2\sigma+2)^2} |\Phi|^{4\sigma}\Phi =0, \quad x \in \R, 
\end{align}
and then the complex-valued function $\phi_{\omega,c}$ satisfies
\[ -\phi ''+ \omega \phi +ic \phi ' - i|\phi|^{2\sigma} \phi '=0, \quad x \in \R.\]
In \cite{LSS13}, they proved that  the solitary waves are orbitally stable if $-2\sqrt{\omega}<c<2z_0 \sqrt{\omega}$, and orbitally unstable if $2z_0\sqrt{\omega}<c<2\sqrt{\omega}$ when $1<\sigma<2$, where the constant $z_0=z_0(\sigma)\in(-1,1)$ is the solution of 
\begin{align*} 
F_{\sigma}(z)&:= (\sigma-1)^2 \l\{ \int_{0}^{\infty} (\cosh y-z)^{-\frac{1}{\sigma}} dy\r\}^2 
\\
&\qquad -  \l\{ \int_{0}^{\infty} (\cosh y-z)^{-\frac{1}{\sigma}-1}(z\cosh y - 1)dy\r\}^2 =0.
\end{align*}
Moreover,  they also proved that the solitary waves for all $\omega>c^2/4$ are orbitally unstable when $\sigma \geq 2$ and orbitally stable when $0<\sigma <1$. Recently, the first author \cite{Fuk16} proved that  the solitary waves are orbitally unstable if $c=2z_0\sqrt{\omega}$ when $7/6<\sigma<2$. More recently, Tang and Xu investigate stability of the two sum of solitary waves for \eqref[gDNLS] (see \cite{TX17pre} for more details).
Before Liu, Simpson, and Sulem \cite{LSS13}, Hao \cite{Hao07} considered \eqref[gDNLS] and proved the local well-posedness in $H^{1/2}(\R)$ when $\sigma\geq 5/2$. Santos \cite{San15} proved the existence and uniqueness of a solution $u \in C ( [ 0 , T ] ; H^{1/2}( \R ) )$ for sufficient small initial data when $\sigma > 1$. 
Recently, the second author and Ozawa \cite{HO16} proved local well-posedness in $H^1(\R)$ when $\sigma \geq 1$, and that the following quantities are conserved.
\begin{align*}
\tag{Energy}
 E(u)&:=\frac{1}{2}\left\Vert \partial_x u  \right\Vert_{L^2}^2 - \frac{1}{2\sigma +2} \re \int_{\R} i |u|^{2\sigma} \overline{u} \partial_x u dx,
\\
\tag{Mass}
 M(u)&:= \left\Vert u \right\Vert_{L^2}^2,
\\
\tag{Momentum}
P(u)&:=\re \int_{\R} i\partial_x u \overline{u} dx.
\end{align*} 
Moreover, they proved global well-posedness for small initial data. They also constructed global solutions for any initial data in $H^1(\mathbb{R})$ in the case $0 <\sigma <1$ ($L^2$-subcritical case). However, in the case $\sigma \geq 1$ ($L^2$-critical or supercritical case), there has been no global existence  result for large data. In the present paper, we investigate global well-posedness for \eqref[gDNLS]  in the case $\sigma \geq 1$ by a variational argument. 
More precisely, we give a variational characterization of solitary waves and a sufficient condition for global existence of solutions to \eqref[gDNLS] by using the characterization. Such an argument was done for nonlinear hyperbolic partial differential equations by Sattinger \cite{S68} (see also \cite{T72,PS75}).
Our argument is also applicable to \eqref[DNLS]. Indeed, the variational argument gives another proof of the result by Wu \cite{Wu15}. Moreover, we prove that the solution of \eqref[DNLS] is global if the initial data $u_0$ satisfies $\norm[u_0]_{L^2}^2=4\pi$ and $P(u_0)<0$. 
\subsection{Main Results}
To state our main results, we introduce some notations.
Let $(\omega,c)$ satisfy 
\begin{equation} 
\label{eq1.5}
\omega > c^2/4, \text{ or } \omega = c^2/4 \text{ and } c>0.
\end{equation}
For $(\omega,c)$ satisfying \eqref[eq1.5], we define 
\begin{equation*}
S_{\omega,c}(\varphi):= E(\varphi)+ \frac{\omega}{2} M(\varphi) + \frac{c}{2} P(\varphi).
\end{equation*}
We denote the nonlinear term by 
\[ N(\varphi):= \re \int_{\R} i |\varphi|^{2\sigma} \overline{\varphi} \partial_x \varphi dx. \]
We define 
\begin{equation*}
\til{S}_{\omega,c}(\psi):=\frac{1}{2}\left\Vert \partial_x \psi \right\Vert_{L^2}^2 +\frac{1}{2}\l( \omega -\frac{c^2}{4}\r) \left\Vert \psi \right\Vert_{L^2}^2 +\frac{c}{2(2\sigma+2)}\left\Vert \psi \right\Vert_{L^{2\sigma+2}}^{2\sigma+2} -\frac{1}{2\sigma+2}N(\psi).
\end{equation*}
Then, we have $S_{\omega,c}(\varphi)=\til{S}_{\omega,c}(e^{-\frac{c}{2}ix}\varphi)$ by using the identities 
\begin{align} 
\label{eq1.6}
c P(\varphi)&=-\left\Vert \partial_x \varphi \right\Vert_{L^2}^2 -\frac{c^2}{4} \left\Vert \varphi \right\Vert_{L^2}^2 +\left\Vert \partial_x (e^{-\frac{c}{2}ix }\varphi) \right\Vert_{L^2}^2, \\
\label{eq1.6.1}
N(\varphi)&=-\frac{c}{2}\left\Vert \varphi \right\Vert_{L^{2\sigma+2}}^{2\sigma+2}+N(e^{-\frac{c}{2}ix}\varphi).
\end{align}
We denote the scaling transformation by $f_{\lambda}^{\alpha,\beta}(x):=e^{\alpha \lambda} f (e^{-\beta \lambda}x)$ for $(\alpha,\beta) \in \R^2$ and any function $f$.  For $(\alpha,\beta) \in \R^2$,  we define
\begin{align*}
\til{K}_{\omega,c}^{\alpha,\beta}(\psi)
&:=\partial_{\lambda} \til{S}_{\omega,c} (\psi_{\lambda}^{\alpha,\beta})|_{\lambda =0},
\\
K_{\omega,c}^{\alpha,\beta}(\varphi)
&:=\til{K}_{\omega,c}^{\alpha,\beta}(e^{-\frac{c}{2}ix}\varphi) .
\end{align*}
By a direct calculation, we have the following explicit formulae.
\begin{align*}
\til{K}_{\omega,c}^{\alpha,\beta}(\psi)
&=\left\langle \til{S}_{\omega,c}'(\psi), \alpha \psi -\beta x  \partial_x \psi \right\rangle
\\
&=\frac{2\alpha-\beta}{2}\left\Vert \partial_x \psi \right\Vert_{L^2}^2 +\frac{2\alpha+\beta}{2} \l( \omega - \frac{c^2}{4} \r)\left\Vert \psi \right\Vert_{L^2}^2 
+\frac{\{(2\sigma+2)\alpha+\beta\} c}{2(2\sigma+2)} \left\Vert \psi \right\Vert_{L^{2\sigma+2}}^{2\sigma+2} -\alpha N(\psi),
\\
K_{\omega,c}^{\alpha,\beta}(\varphi)
&=\left\langle \til{S}_{\omega,c}'(e^{-\frac{c}{2}ix}\varphi), \alpha e^{-\frac{c}{2}ix}\varphi -\beta x  \partial_x(e^{-\frac{c}{2}ix}\varphi)  \right\rangle
\\
&=\left\langle S_{\omega,c}'(\varphi) ,\alpha \varphi+\frac{c}{2}i\beta x \varphi -\beta x  \partial_x\varphi  \right\rangle
\\
&= \frac{2\alpha-\beta}{2}\left\Vert \partial_x \varphi \right\Vert_{L^2}^2 +\l( \frac{2\alpha+\beta}{2} \omega - \frac{c^2}{4} \beta \r)\left\Vert \varphi \right\Vert_{L^2}^2 +\frac{2\alpha-\beta}{2}c P(\varphi) 
\\
& \quad+\frac{\beta c}{2(2\sigma+2)} \left\Vert \varphi \right\Vert_{L^{2\sigma+2}}^{2\sigma+2} -\alpha N(\varphi),
\end{align*}
where we have used \eqref[eq1.6] and \eqref[eq1.6.1]. 

\begin{remark}
\ 

\begin{enumerate}
\item If $\beta\neq 0$, then $K_{\omega,c}^{\alpha,\beta}$ is different from $I_{\omega,c}^{\alpha,\beta}(\varphi):=\partial_{\lambda} S_{\omega,c} (\varphi_{\lambda}^{\alpha,\beta})|_{\lambda =0}$. Indeed, the explicit formula of $I_{\omega,c}^{\alpha,\beta}$ is as follows.
\[ I_{\omega,c}^{\alpha,\beta}(\varphi)
= \frac{2\alpha-\beta}{2}\norm[\partial_x \varphi]_{L^2}^2 +\frac{2\alpha+\beta}{2} \omega \norm[\varphi]_{L^2}^2 
+c\alpha P(\varphi)  -\alpha N(\varphi). \] 
We note that $K_{\omega,c}^{\alpha,0}$ coincides with $I_{\omega,c}^{\alpha,0}$, and especially $K_{\omega,c}^{1,0}=I_{\omega,c}^{1,0}$ is nothing but  the Nehari functional.
\item It is not clear whether the momentum $P$ is positive or not. That is why we introduce $\til{S}_{\omega,c}$ by using \eqref[eq1.6]. Such an argument can be seen in \cite{BGL14} (see the equation (14) in \cite{BGL14} for the detail). 
\item The functional $K_{\omega,c}^{\alpha,\beta}$  is more useful to obtain the characterization of the solitary waves when $\omega=c^2/4$ and $c>0$ than $I_{\omega,c}^{\alpha,\beta}$ since $K_{\omega,c}^{\alpha,\beta}$ contains the $L^{2\sigma+2}$-norm (see the proof in \secref[sec2.2]). 
\item $\til{S}_{\omega,c}$ and $\til{K}_{\omega,c}^{\alpha,\beta}$ are relevant to the following elliptic equation.
\[ - \psi ''+\l(\omega - \frac{c^2}{4}\r) \psi + \frac{c}{2} |\psi|^{2\sigma}\psi - i|\psi|^{2\sigma} \psi '=0, \quad x \in \R. \]
\end{enumerate}
\end{remark}

We define the following function space for $(\omega,c)$ satisfying \eqref[eq1.5]. 
\begin{equation*}
X_{\omega,c}:=\l\{
\begin{array}{ll}
H^1(\R), & \text{if } \omega>c^2/4,
\\
\dot{H}^1(\R) \cap L^{2\sigma+2}(\R), & \text{if } \omega=c^2/4 \text{ and } c>0.
\end{array}
\r.
\end{equation*}
We consider the following minimization problem:
\begin{align*}
\mu_{\omega,c}^{\alpha,\beta}&:=\inf\{ S_{\omega,c}(\varphi): e^{-\frac{c}{2}ix}\varphi \in X_{\omega,c}\setminus\{0\}, K_{\omega,c}^{\alpha,\beta}(\varphi)=0 \}
\\
&=\inf\{ \til{S}_{\omega,c}(\psi):\psi \in X_{\omega,c}\setminus\{0\}, \til{K}_{\omega,c}^{\alpha,\beta}(\psi)=0 \}.
\end{align*}

\begin{remark}\ 

\begin{enumerate}
\item We note that the solitary waves $\phi_{c^2/4,c}$ do not belong to $L^2(\R)$ when $\sigma \geq 2$. Therefore, we define $X_{c^2/4,c}:=\dot{H}^1(\R)\cap L^{2\sigma+2}(\R)$ to characterize the solitary waves $\phi_{c^2/4,c}$ (cf \cite{KW16pre}). 
\item $S_{c^2/4,c}$ seems meaningless on the function space
$\{\varphi: e^{-\frac{c}{2}ix}\varphi \in X_{c^2/4,c}\}$
since $S_{c^2/4,c}$ contains $L^2$-norm. However, in fact,  $S_{c^2/4,c}$ is well-defined on the function space
since $\til{S}_{c^2/4,c}$ is defined on $\dot{H}^1(\R) \cap L^{2\sigma+2}(\R)$ and the equality $S_{c^2/4,c}(\varphi)=\til{S}_{c^2/4,c}(e^{-\frac{c}{2}ix}\varphi)$ holds. 
Similarly, $K_{c^2/4,c}^{\alpha,\beta}$ is well-defined on this function space.
\item
Since $\varphi \in H^1(\R)$ if and only if $e^{-\frac{c}{2}ix}\varphi \in H^1(\R)$, when $\omega>c^2/4$, we have
\[ \mu_{\omega,c}^{\alpha,\beta}=\inf\{ S_{\omega,c}(\varphi): \varphi \in H^1(\R) \setminus\{0\}, K_{\omega,c}^{\alpha,\beta}(\varphi)=0 \}. \] 
However, when $\omega=c^2/4$ and $c>0$, the above equality does not hold.
\end{enumerate}
\end{remark}

We assume that $(\alpha,\beta)\in \R^2$ satisfies 
\begin{equation}
\label{eq1.7}
\l\{
\begin{array}{ll}
2\alpha-\beta>0,\  2\alpha+\beta>0, \text{ and }  \beta c \leq 0,  & \text{when } \omega>c^2/4,
\\
2\alpha-\beta> 0,\  2\alpha+\beta> 0, \text{ and }  \beta < 0,  & \text{when } \omega=c^2/4 \text{ and } c>0.
\end{array}
\r.
\end{equation}

We define some function spaces. 
\begin{align*}
\scM_{\omega,c}^{\alpha,\beta}&:=\{\varphi: e^{-\frac{c}{2}ix} \varphi \in X_{\omega,c} \setminus\{0\}, S_{\omega,c}(\varphi)=\mu_{\omega,c}^{\alpha,\beta},\ K_{\omega,c}^{\alpha,\beta}(\varphi)=0  \},
\\
\scG_{\omega,c}&:=\{\varphi: e^{-\frac{c}{2}ix} \varphi \in X_{\omega,c} \setminus\{0\}, S_{\omega,c}'(\varphi)=0 \}.
\end{align*}

We give the following characterization of the solitary waves.
\begin{theorem}
\label{thm1.1}
Let $\sigma \geq 1$, $(\omega,c)$ satisfy \eqref[eq1.5], and $(\alpha,\beta)$ satisfy \eqref[eq1.7]. Then, we have 
\[ \scM_{\omega,c}^{\alpha,\beta}=\scG_{\omega,c}=\{e^{i\theta_0} \phi_{\omega,c}(\cdot-x_0): \theta_0 \in [0,2\pi), x_0\in \R \}. \]
\end{theorem}


\thmref[thm1.1] also means that $\mu_{\omega,c}^{\alpha,\beta}$ and $\scM_{\omega,c}^{\alpha,\beta}$ are independent of $(\alpha,\beta)$ and $\scM_{\omega,c}^{\alpha,\beta}$ is not empty. Thus, we denote $\mu_{\omega,c}^{\alpha,\beta}$ by $\mu_{\omega,c}$.

We define 
\begin{align*}
	\scK_{\omega,c}^{\alpha,\beta,+}&:=\{ \varphi \in H^1(\R):S_{\omega,c}(\varphi)\leq \mu_{\omega,c}, K_{\omega,c}^{\alpha,\beta}(\varphi)\geq 0 \},
	\\
	\scK_{\omega,c}^{\alpha,\beta,-}&:=\{ \varphi \in H^1(\R):S_{\omega,c}(\varphi)\leq \mu_{\omega,c}, K_{\omega,c}^{\alpha,\beta}(\varphi)< 0 \}.
\end{align*}
The characterization by \thmref[thm1.1] gives us the following sufficient condition for global existence. 
\begin{theorem}
\label{thm1.2}
Let $\sigma\geq 1$, $(\omega,c)$ satisfies \eqref[eq1.5], and $(\alpha,\beta)$ satisfies \eqref[eq1.7]. Then, $\scK_{\omega,c}^{\alpha,\beta,\pm}$ are invariant under the flow of \eqref[gDNLS]. Namely, if the initial data $u_0$ belongs to $\scK_{\omega,c}^{\alpha,\beta,\pm}$, then the solution $u(t)$ of \eqref[gDNLS] also belongs to $\scK_{\omega,c}^{\alpha,\beta,\pm}$ for all $t\in I_{\max}$, where $I_{\max}$ denotes the maximal existence time. 

Moreover, if the initial data $u_0$ belongs to $\scK_{\omega,c}^{\alpha,\beta,+}$ for some $(\omega,c)$ satisfying \eqref[eq1.5] and $(\alpha,\beta)$ satisfying \eqref[eq1.7], then the corresponding solution $u$ of \eqref[gDNLS] exists globally in time and we have
\[ \norm[u]_{L^{\infty}(\R:H^1(\R))}\leq C(\norm[u_0]_{H^1}), \]
where $C:[0,\infty)\to \R$ is continuous.  
\end{theorem}

Recently, Miao, Tang, and Xu independently obtained the similar results as Theorems \ref{thm1.1} and \ref{thm1.2} when $\sigma=1$ in \cite{MTX16pre_1}. We will compare their method with our argument in \remref[rem1.5].

We show that \thmref[thm1.2] gives us some interesting corollaries for (DNLS). 

\begin{corollary}
\label{cor1.3}
Let $\sigma=1$. If the initial data $u_0 \in H^1(\R )$ satisfies $\norm[u_0]_{L^2}^2 < 4\pi$, then the solution of \eqref[DNLS] is global.
\end{corollary}

Two proofs have been known for Corollary \ref{cor1.3}. One is obtained by Wu \cite{Wu15} and another one is obtained by Guo and Wu \cite{GW17}. We give another proof by \thmref[thm1.2].
We compare the methods of Wu \cite{Wu15} and Guo and Wu \cite{GW17}, which depend on the sharp Gagliardo--Nirenberg type inequality, with our variational argument. 
%
Using the following gauge transformation to the solution of \eqref[DNLS]
\begin{align}
\label{gauge}
u(t,x) = w(t,x)\exp\l( -\frac{i}{4} \int_{-\infty}^{x} |w(t,x)|^2 dx\r), 
\end{align}
then $w$ satisfies the following equation.
\begin{align}
\label{DNLS1}
\l\{
\begin{array}{ll}
i \partial_t w + \partial_x^2 w +\frac{i}{2} |w|^{2} \partial_x w -\frac{i}{2}w^2\partial_x \overline{w} +\frac{3}{16}|w|^4w =0 , & (t,x) \in \R \times \R, \\
w(0,x) = w_0(x), & x \in \R .
\end{array}
\r.
\end{align}
The energy and the momentum are transformed as follows.
\begin{align*}
&  {\cE}(w)= \frac{1}{2}\norm[\partial_x w]_{L^2}^2  -\frac{1}{32} \norm[w]_{L^6}^6 ,\\
&  {\cP}(w)= \re \int_{\R} i\partial_x w \overline{w} dx +\frac{1}{4}\norm[w]_{L^4}^4 .
\end{align*}
Hayashi and Ozawa \cite{HO92} used the following sharp Gagliardo--Nirenberg inequality
\begin{align}
\label{GN1}
\| f\|_{L^6}^6 \leq \frac{4}{\pi^2}\| f\|_{L^2}^4\| \partial_{x}f\|_{L^2}^2
\end{align}
in order to obtain a priori estimate in $\dot{H}^1(\R )$. We note that the optimizer for the inequality \eqref[GN1] is given by $Q:= \Phi_{1,0}$ and $Q$ satisfies the following elliptic equation.
\begin{align}
\label{ellipQ}
-Q''+Q -\frac{3}{16} Q^5 =0 .
\end{align}
In \cite{HO92}, they proved the $H^1$-solution of \eqref[DNLS] is global if the initial data $u_0$ satisfies $\| u_0 \|_{L^2}^2 =\| w_0 \|_{L^2}^2 < \| Q \|_{L^2}^2=2\pi $ (see also Weinstein \cite{W83}).
Wu \cite{Wu15} used not only the energy but also the momentum, and the following sharp Gagliardo--Nirenberg inequality
\begin{align}
\label{GN2}
\| f\|_{L^6}^6 \leq 3(2\pi )^{-\frac{2}{3}}\| f\|_{L^4}^{\frac{16}{3}}\| \partial_{x}f\|_{L^2}^{\frac{2}{3}}.
\end{align}
We note that the optimizer for the inequality \eqref[GN2] is given by $W:= \Phi_{\frac{1}{4},1}$ and $W$ satisfies the following elliptic equation.
\begin{align}
-W'' +\frac{1}{2}W^3 -\frac{3}{16} W^5 =0.
\end{align}
Wu \cite{Wu15} proved that the $H^1$-solution of \eqref[DNLS] is global if the initial data $u_0$ satisfies $\| u_0 \|_{L^2}^2 =\| w_0 \|_{L^2}^2 < \| W \|_{L^2}^2=4\pi $. His proof depends on contradiction argument. Supposing that there exists a time sequence $\{t_n\}_{n\in \N}$ with $t_n \to T_{\max}$, or $-T_{\min}$ such that $\norm[\partial_x w(t_n)]_{L^2} \to \infty$ as $n \to \infty$, where $(-T_{\min},T_{\max})$ is the maximal time interval, he mainly proved that $X=\norm[v(t_n)]_{L^4}^{8}/\norm[v(t_n)]_{L^6}^{6}$ satisfies $X^3-\norm[v]_{L^2}^2X^2+16\{3(2\pi )^{-\frac{2}{3}}\}^{-3}\norm[v]_{L^2}^2<0$, but this does not hold when $\norm[v]_{L^2}^2<4\pi$. 
On the other hand, Guo and Wu \cite{GW17} gave a priori estimate directly for \eqref[DNLS1] by the sharp Gagliardo--Nirenberg inequality \eqref[GN2]. More precisely, they showed in \cite[Lemma 2.1]{GW17} the inequality
\begin{align}
\label{eq1.15}
\cP(w) \geq \frac{1}{4} \norm[w]_{L^4}^4 \l( 1-\frac{\norm[w]_{L^2}}{2\sqrt{\pi}}\r) - \frac{8\sqrt{\pi}\cE(w)\norm[w]_{L^2}}{\norm[w]_{L^4}^4},
\end{align}
and thus $\Vert \partial_x w \Vert_{L^2}^2$ is bounded by $\cP$ and $\cE$ if $\Vert w \Vert_{L^2}^2<4\pi$ (see \cite[Lemma 2.2]{GW17}).
In our variational argument, we do not use contradiction argument, the gauge transformation like \eqref[gauge], and any sharp Gagliardo--Nirenberg inequality. 


We give the global existence result in the threshold case by \thmref[thm1.2]. 

\begin{corollary}
\label{cor1.4}
Let $\sigma=1$. We assume that the initial data $u_0 \in H^1(\R )$ satisfies $\norm[u_0]_{L^2}^2 = 4\pi$. If $P(u_0)<0$, then the solution of \eqref[DNLS] is global.
\end{corollary} 

After submitting the present paper, Guo pointed out that Corollary \ref{cor1.4} can be obtained by \eqref[eq1.15]. We also give the proof by \eqref[eq1.15] for the reader's convenience.

The following corollary means that there exist global solutions with any large mass.

\begin{corollary}
\label{cor1.5}
Let $\sigma \geq 1$. Given $\psi \in H^1(\mathbb{R})$, and set the initial data as $u_{0,c}=e^{\frac{c}{2}ix}\psi$. Then, there exists $c_0 >0$ such that if $c\geq c_0$, then the corresponding solution $u_c$ of \eqref[gDNLS] is global. 
\end{corollary}


\begin{remark}
The existence of blow-up solutions in finite time is still an open problem. It might be a very interesting problem whether finite time blow-up occurs when the initial data $u_0$ satisfies $\| u_0\|_{L^2}^2 =4\pi$ and $P(u_0) >0$.
\end{remark}
\subsection{Compare DNLS with mass-critical NLS}

The equation \eqref[DNLS] is $L^2$-critical in the sense that the equation and $L^2$-norm are invariant under the scaling transformation
\[ u_{\gamma}(t,x):=\gamma^{\frac{1}{2}} u(\gamma^2 t,\gamma x), \quad \gamma>0. \]
The same invariance holds for the quintic nonlinear Schr\"{o}dinger equation in one dimensional space:
\begin{equation}
\label{NLS}
i \partial_t u +\partial_x^2 u + \frac{3}{16}|u|^{4}u=0, \quad (t,x)\in \R \times \R.
\end{equation} 
This equation has the same energy of the equation \eqref[DNLS1]. It is known that \eqref[NLS] is locally well-posed in the energy space $H^1(\R)$ and the solution is global if the initial data $u_0$ satisfies $\norm[u_0]_{L^2}^2 < \norm[Q]_{L^2}^2$, where $Q$ is the ground state of the same elliptic equation \eqref[ellipQ]. 
The condition $\norm[u_0]_{L^2}^2 < \norm[Q]_{L^2}^2$ is equivalent to the condition obtained by the variational argument. In this argument, the momentum is not essential since the equation \eqref[NLS] is invariant under the Galilean transformation and thus we may assume that the momentum is zero. On the other hand, \eqref[DNLS] is not invariant under the Galilean transformation. Therefore, the condition by the variational argument is better than the assumption $\norm[u_0]_{L^2}^2 <\norm[W]_{L^2}^2= 4\pi$. Indeed, the momentum and the parameter $c$ play important roles in Corollaries \ref{cor1.4} and \ref{cor1.5}.

\subsection{Idea of Proofs}
\label{sec1.4}

The proof of \thmref[thm1.1] is based on the method of Colin and Ohta \cite{CO06} (concentration compactness method). They characterized the solitary waves for $\omega>c^2/4$ when $\sigma =1$ by the Nehari functional $I_{\omega,c}^{1,0}$. However, in the case $\omega=c^2/4$ and $c>0$, we cannot apply their argument directly since the $L^2$-norm in $I_{\omega,c}^{1,0}$ disappears by \eqref[eq1.6]. Therefore, we introduce the new functional $K_{\omega,c}^{\alpha,\beta}$ for $(\alpha,\beta)$ satisfying \eqref[eq1.7]. We can use the  $L^{2\sigma+2}$-norm instead of the $L^2$-norm by using $K_{\omega,c}^{\alpha,\beta}$. That is why we introduce the function space $X_{\omega,c}$ as $\dot{H}^1\cap L^{2\sigma+2}$ in the massless case (i.e. $\omega=c^2/4$ and $c>0$). Noting that the solitary waves $\phi_{c^2/4,c}$ do not belong to $L^2(\R)$ when $\sigma\geq2$, the function space $X_{\omega,c}$ is essential to obtain the characterization of the solitary waves $\phi_{c^2/4,c}$ (cf. \cite{KW16pre}). Based on the argument of Colin and Ohta \cite{CO06}, we characterize the solitary waves $\phi_{c^2/4,c}$ by $K_{\omega,c}^{\alpha,\beta}$. By the conservation laws and the characterization of the solitary waves, we get a priori estimate and thus we obtain  \thmref[thm1.2]. The corollaries follow from \thmref[thm1.2]. In their proof, the parameter $c$ plays an important role. More precisely, taking $c>0$ large, we get the corollaries. At last, we emphasize that we do not use the sharp Gagliardo--Nirenberg inequality and we do not apply the gauge transformation to \eqref[gDNLS] since the equation after applying the transformation is complicated unlike \eqref[DNLS]. 

\begin{remark}
\label{rem1.5} 
Miao, Tang, and Xu \cite{MTX16pre_1} treated the case of $\sigma=1$. 
They considered \eqref[DNLS1] by using the gauge transformation and defined the action by $\cS_{\omega,c}:=\cE+ \omega M/2 + c\cP/2$. They applied concentration compactness argument to give the variational characterization of the solitary waves. Then, they use the Nehari functional $\cK_{\omega,c}$ derived from the action  $\cS_{\omega,c}$. The explicit formula of $\cK_{\omega,c}$ is 
\[ \cK_{\omega,c}(w):=\norm[\partial_x w]_{L^2}^2 -\frac{3}{16} \norm[w]_{L^6}^6 + \omega \norm[w]_{L^2}^2 + c \re \int_{\R} i \partial_x w \overline{w} dx +\frac{c}{2}\norm[w]_{L^4}^4. \]
They defined 
\[ \scA_{\omega,c}^{\pm}:=\l\{ \varphi \in H^1(\R):\cS_{\omega,c}(\varphi)\leq \cS_{\omega,c}(\phi_{\omega,c}), \cK_{\omega,c}(\varphi) \gtreqless 0 \r\} \]
and they also showed that $\scA_{\omega,c}^{\pm}$ are invariant under the flow of \eqref[DNLS1] and the solution to \eqref[DNLS1] is global if $w_0 \in \scA_{\omega,c}^{+}$ for some $(\omega,c)$.  The functional $\cK_{\omega,c}$ is useful to characterize the solitary waves $\phi_{c^2/4,c}$ since it contains $L^4$-norm. Namely, one can use the Nehari functional by the gauge transformation. On the other hand, we cannot use Nehari functional when we do not apply the gauge transformation, and thus we introduce the new functionals $K_{\omega,c}^{\alpha,\beta}$.
\end{remark}

The rest of the present paper is as follows. In \secref[sec2.1], we prepare some lemmas to obtain the characterization of the solitary waves and prove a priori estimate (see \eqref[eq2.2]). In \secref[sec2.2], we give the characterization of the solitary waves $\phi_{c^2/4,c}$. We remark that the characterization of the solitary waves $\phi_{\omega,c}$ for $\omega>c^2/4$ can be obtained in the same manner as in Colin and Ohta \cite{CO06} and then we omit the proof. \secref[sec3] is devoted to the proof of \thmref[thm1.2] and the corollaries. In \appref[appA], we show that there is no non-trivial solution of the nonlinear elliptic equation \eqref[elliptic] if $\omega<c^2/4$, or $\omega=c^2/4$ and $c\leq 0$.

\section{Variational Characterization of the solitary waves}

\subsection{Preliminaries}
\label{sec2.1}
We define function spaces
\begin{align*}
\til{\scM}_{\omega,c}^{\alpha,\beta}&:=\{\psi \in X_{\omega,c}\setminus\{0\}: \til{S}_{\omega,c}(\psi)=\mu_{\omega,c}^{\alpha,\beta},\ \til{K}_{\omega,c}^{\alpha,\beta}(\psi)=0  \},
\\
\til{\scG}_{\omega,c}&:=\{\psi \in X_{\omega,c}\setminus\{0\}: \til{S}_{\omega,c}'(\psi)=0 \}.
\end{align*}

In this section, we prove the following proposition, which gives \thmref[thm1.1].

\begin{proposition}
\label{prop2.1}
Let $(\omega,c)$ satisfy \eqref[eq1.5] and $(\alpha,\beta)$ satisfy \eqref[eq1.7]. Then, we have 
\[ \til{\scM}_{\omega,c}^{\alpha,\beta}=\til{\scG}_{\omega,c}=\{e^{i\theta} e^{-\frac{c}{2}ix} \phi_{\omega,c}(\cdot-y): \theta \in [0,2\pi), y\in \R \}. \]
\end{proposition}

Indeed, \thmref[thm1.1] follows from \propref[prop2.1] and the following properties:
\begin{align*}
\varphi \in \scM_{\omega,c}^{\alpha,\beta} &\Leftrightarrow e^{-\frac{c}{2}ix} \varphi \in \til{\scM}_{\omega,c}^{\alpha,\beta},
\\
\varphi \in  \scG_{\omega,c} &\Leftrightarrow e^{-\frac{c}{2}ix} \varphi \in \til{\scG}_{\omega,c},
\end{align*}
where we note that $\til{S}_{\omega,c}'(e^{-\frac{c}{2}ix}\varphi)=e^{-\frac{c}{2}ix}S_{\omega,c}'(\varphi)$ holds. 

To prove \propref[prop2.1], we prepare some basic lemmas. We have the Gagliardo--Nirenberg type inequality. 
\begin{lemma}
Let $p\geq1$. 
We have the following estimate.
\begin{align}
\label{eq2.1}
\norm[f]_{L^\infty}^{2p}
& \leq 2p\norm[f]_{L^{4p-2}}^{2p-1}\norm[\partial_x f]_{L^2}.
\end{align}
\end{lemma}

\begin{proof}
By the H\"{o}lder inequlity, we get
\begin{align*}
|f(x)|^{2p}
& =\int_{-\infty}^{x} \frac{d}{dx} (|f(y)|^{2p}) dy
\\
&=\int_{-\infty}^{x} 2p |f(y)|^{2p-2} \re (\overline{f(y)} (\partial_x f)(y))dy
\\
& \leq 2p \norm[|f|^{2p-1}]_{L^2} \norm[\partial_x f]_{L^2}
\\
& = 2p \norm[f]_{L^{4p-2}}^{2p-1} \norm[\partial_x f]_{L^2}.
\end{align*}
Taking the supremum, we obtain \eqref[eq2.1]. 
\end{proof}

We have the Lieb compactness lemma. See \cite{Lie83} for $p=2$ and \cite[Lemma 2.1]{BFV14} for more general setting. 
\begin{lemma}
\label{lem2.3}
Let $p\geq 2$ and $d\in\N$. 
Let $\{f_n\}$ be a bounded sequence in $\dot{H}^1(\R^d)\cap L^{p}(\R^d)$. Assume that there exists $q\in(p,2^*)$ such that $\limsup_{n \to \infty} \norm[f_n]_{L^q}>0$. Then, there exists $\{y_n\}$ and $f \in \dot{H}^1(\R^d)\cap L^{p}(\R^d)\setminus \{0\}$ such that $\{f_n(\cdot-y_n)\}$ has a subsequence that converges to $f$ weakly in $\dot{H}^1(\R^d)\cap L^{p}(\R^d)$. 
\end{lemma}

We have the Brezis--Lieb lemma (see \cite{BL83}). 
\begin{lemma}
\label{lem2.4}
Let $d\in \N$ and $1< p < \infty$. Let $\{f_n\}$ be a bounded sequence in $L^p(\R^d)$ and $f_n \to f$ a.e. in $\R^d$. Then we have 
\[ \norm[f_n]_{L^p}^p - \norm[f_n-f]_{L^p}^p - \norm[f]_{L^p}^p \to 0. \]
If $\{f_n\}$ is a bounded sequence in $L^2(\R^d)$ and $f_n$ converges to $f$ weakly in $L^2(\R^d)$, then the statement as $p=2$ holds. 
\end{lemma}

A direct calculation gives us the following relation. 
\begin{lemma} We have
\begin{align}
\label{eq2.2}
\alpha (2\sigma+2) \til{S}_{\omega,c}(\psi)
&=\til{K}_{\omega,c}^{\alpha,\beta}(\psi)
+\frac{2\sigma\alpha+\beta}{2}\norm[\partial_x \psi]_{L^2}^2 
\\ \notag
&\quad +\l( \omega-\frac{c^2}{4}\r)\frac{2\sigma\alpha-\beta}{2}\norm[\psi]_{L^2}^2 -\frac{\beta c}{2(2\sigma+2)}\norm[\psi]_{L^{2\sigma+2}}^{2\sigma+2}.
\end{align}
\end{lemma}

We denote the difference $\alpha (2\sigma+2) \til{S}_{\omega,c}(\psi)- \til{K}_{\omega,c}^{\alpha,\beta}(\psi)$ by
\[  \til{J}_{\omega,c}^{\alpha,\beta}(\psi):=\frac{2\sigma\alpha+\beta}{2}\norm[\partial_x \psi]_{L^2}^2 
+\l( \omega-\frac{c^2}{4}\r)\frac{2\sigma\alpha-\beta}{2}\norm[\psi]_{L^2}^2 -\frac{\beta c}{2(2\sigma+2)}\norm[\psi]_{L^{2\sigma+2}}^{2\sigma+2}.\]

\subsection{Variational Characterization}
\label{sec2.2}
First, we consider the case of $\omega=c^2/4$ and $c>0$. 
Then, $(\alpha,\beta)$ satisfies 
\begin{equation}
\label{eq2.3}
2\alpha-\beta>0 , \quad 2\alpha+\beta>0, \text{ and } \beta < 0. 
\end{equation}
In the sequel section, we often omit the indices $\omega,c,\alpha,\beta$ for simplicity.

\begin{lemma} 
\label{lem2.6.1}
The following equality holds.
\[ \til{\scG}_{\omega,c}=\{e^{i\theta_0}e^{-\frac{c}{2}ix} \phi_{\omega,c}(\cdot-x_0): \theta_0 \in [0,2\pi), x_0\in \R \}.\]
\end{lemma}

\begin{proof}
Since $e^{-\frac{c}{2}ix}\phi_{\omega,c}$ satisfies $\til{S}_{\omega,c}'(e^{-\frac{c}{2}ix}\phi_{\omega,c})=e^{-\frac{c}{2}ix}S_{\omega,c}'(\phi_{\omega,c})=0$, we have
$\til{\scG}_{\omega,c} \supset \{e^{i\theta_0}e^{-\frac{c}{2}ix} \phi_{\omega,c}(\cdot-x_0): \theta_0 \in [0,2\pi), x_0\in \R \}$.
We prove $\til{\scG}_{\omega,c} \subset \{e^{i\theta_0}e^{-\frac{c}{2}ix} \phi_{\omega,c}(\cdot-x_0): \theta_0 \in [0,2\pi), x_0\in \R \}$.
Letting $\psi \in \til{\scG}_{\omega,c}$ and 
\[ \psi(x) 
= \Phi(x) \exp\l( - \frac{i}{2\sigma +2} \int_{0}^{x} |\Phi(y)|^{2\sigma} dy \r),\]
then $\Phi$ is a solution of
\[ -\Phi''  +\frac{c}{2} |\Phi|^{2\sigma} \Phi - \frac{2\sigma+1}{(2\sigma+2)^2} |\Phi|^{4\sigma}\Phi + \frac{\sigma}{\sigma+1} |\Phi|^{2\sigma-2} \im (\overline{\Phi}\Phi')\Phi=0. \]
Setting $A(\Phi):=\frac{c}{2} |\Phi|^{2\sigma} - \frac{2\sigma+1}{(2\sigma+2)^2} |\Phi|^{4\sigma}+ \frac{\sigma}{\sigma+1} |\Phi|^{2\sigma-2} \im (\overline{\Phi}\Phi')$, $f:=\re \Phi$, and $g:=\im \Phi$, we have
\[ f ''=A(\Phi)f, \quad g ''=A(\Phi) g.\] 
Therefore, 
\begin{align*}
(f g' - g f' ) '
= f g '' - g f''
=f A(\Phi) g - g A(\Phi) f
=A(\Phi) f g - A(\Phi) fg
=0.
\end{align*}
Since $f,g \in \dot{H}^1(\R)\cap L^{2\sigma+2}(\R)$, we obtain $f  g ' - g f ' =0$. On the other hand, $f g ' - g f ' = \re \Phi \im \Phi' - \im \Phi \re \Phi' 
= \im (\overline{\Phi}\Phi')$. Thus, $\im (\overline{\Phi}\Phi')= 0$ for any $x \in \R$. Therefore, $\Phi$ satisfies 
\begin{equation}  
\label{eq2.4}
-\Phi''  +\frac{c}{2} |\Phi|^{2\sigma} \Phi - \frac{2\sigma+1}{(2\sigma+2)^2} |\Phi|^{4\sigma}\Phi =0.
\end{equation}
Therefore, there exist $\theta_0$ and $x_0$ such that $\Phi=e^{i\theta_0} \Phi_{\omega,c}(\cdot-x_0)$ since $\Phi_{\omega,c}$ is the unique solution of \eqref[eq2.4] up to translation and phase (see Appendix A). 
This implies $\psi(x) =e^{i\theta}e^{-\frac{c}{2}ix} \phi_{\omega,c}(x-x_0)$.
\end{proof}

\begin{remark}
According to \cite{CO06}, it looks natural to take the integral on the infinite interval $(-\infty,x]$ in the gauge transformation as follows
\[ \psi(x) = \Phi(x) \exp\l( - \frac{i}{2\sigma +2} \int_{-\infty}^{x} |\Phi(y)|^{2\sigma} dy \r). \]
However, in the massless case, it is not clear whether $\psi \in \til{\scG}_{\omega,c}$ belongs to $L^{2\sigma}(\R)$ or not. This is why we take the integral on the finite interval $[0,x]$ instead of $(-\infty,x]$. 
\end{remark}

\begin{lemma}
\label{lem2.7.1}
We have $\til{\scG}_{\omega,c} \supset \til{\scM}_{\omega,c}^{\alpha,\beta}$. 
\end{lemma}

\begin{proof}
This is obvious if $\til{\scM}= \emptyset$. We consider the case of $\til{\scM} \neq \emptyset$. Let $\psi \in  \til{\scM}$. Since $\psi$ is a minimizer, there exists a Lagrange multiplier $\eta \in \R$ such that $\til{S}'(\psi)=\eta \til{K}'(\psi)$. 
Then,  we get 
\begin{align*} 
0=  \til{K}(\psi)
=\tbra[\til{S}'(\psi),\partial_\lambda  \psi_{\lambda}^{\alpha,\beta}|_{\lambda=0}]
=\eta \tbra[\til{K}'(\psi), \partial_\lambda \psi_{\lambda}^{\alpha,\beta}|_{\lambda=0}]
=\eta \partial_{\lambda} \til{K}(\psi_{\lambda}^{\alpha,\beta})|_{\lambda=0},
\end{align*}
where we remark that this is justified by a density argument. 
By  a direct calculation, we obtain
\begin{align*}
&\partial_{\lambda}\til{K}(\psi_{\lambda}^{\alpha,\beta})|_{\lambda=0}
\\
&= \frac{(2\alpha-\beta)^2}{2} \norm[\partial_x \psi]_{L^2}^2 - \frac{\{(2\sigma+2)\alpha+\beta\}^2}{2(2\sigma+2)}\norm[\psi]_{L^{2\sigma+2}}^{2\sigma+2} -\frac{\{(2\sigma+2)\alpha\}^2}{2\sigma+2}N(\psi)
\\
&=\frac{-(2\alpha-\beta)(2\sigma \alpha +\beta)}{2}\norm[\partial_x \psi]_{L^2}^2
+\frac{\{(2\sigma+2)\alpha+\beta\}\beta c}{2(2\sigma+2)}\norm[\psi]_{L^{2\sigma+2}}^{2\sigma+2}
+ (2\sigma+2)\alpha \til{K}(\psi)
\\
&<0,
\end{align*}
where in the last inequality we use 
\[ 2\alpha-\beta>0, 2 \alpha +\beta>0, \beta<0, \text{ and } \til{K}(\psi)=0. \]
Therefore, $\eta=0$. This implies $\til{S}_{\omega,c}'(\psi)=0$ and then $\psi \in \til{\scG}_{\omega,c}$.
\end{proof}

\begin{lemma}
We have
$\til{\scG}_{\omega,c} \subset \til{\scM}_{\omega,c}^{\alpha,\beta}$ if $\til{\scM}_{\omega,c}^{\alpha,\beta} \neq \emptyset$. 
\end{lemma}

\begin{proof}
 Let $\psi \in \til{\scG}$. Then, there exist $\theta_0 \in [0,2\pi)$ and $x_0 \in \R$ such that $\psi=e^{i\theta_0}e^{-\frac{c}{2}ix} \phi_{\omega,c}(\cdot-x_0)$ by \lemref[lem2.6.1]. If $\til{\scM} \neq \emptyset$, then we can take $\varphi \in \til{\scM}$. By Lemmas \ref{lem2.6.1} and \ref{lem2.7.1}, there exist $\theta_1 \in [0,2\pi)$ and $x_1\in \R$ such that $\varphi=e^{i\theta_1}e^{-\frac{c}{2}ix} \phi_{\omega,c}(\cdot-x_1)$. Thus, 
$\til{S}_{\omega,c}(\psi)=\til{S}_{\omega,c}(\phi_{\omega,c})=\til{S}_{\omega,c}(\varphi)=\mu_{\omega,c}$.
Moreover, we have $\til{K}(\psi)=\tbra[\til{S}_{\omega,c}'(\psi),\partial_\lambda \psi_{\lambda}^{\alpha,\beta}|_{\lambda=0}]=0$.
\end{proof}

\begin{lemma}
We have
$\til{\scM}_{\omega,c}^{\alpha,\beta} \neq \emptyset$. 
\end{lemma}

To prove this lemma, we show the following proposition. 

\begin{proposition}
\label{prop2.10}
Let $\{\psi_n\}_{n\in\N} \subset X_{\omega,c}$ satisfy 
\[ \til{S}_{\omega,c}(\psi_n) \to \mu_{\omega,c}^{\alpha,\beta} \text{ and } \til{K}_{\omega,c}^{\alpha,\beta}(\psi_n) \to 0. \]
Then, there exists $\{y_n\} \subset \R$ and $\psi \in \til{\scM}_{\omega,c}^{\alpha,\beta}$ such that $\{\psi_n(\cdot - y_n)\}$ has a subsequence which converges to $\psi$ strongly in $X_{\omega ,c}$. 
\end{proposition}
To prove this proposition, at first, we prove the following lemma. 
\begin{lemma}
We have $\mu_{\omega,c}^{\alpha,\beta}>0$.
\end{lemma}

\begin{proof}
We recall that $\mu_{\omega,c}^{\alpha,\beta}=\inf\{ \til{S}_{\omega,c}(\psi):\psi \in X_{\omega,c}\setminus\{0\}, \til{K}_{\omega,c}^{\alpha,\beta}(\psi)=0 \}$. 
By \eqref[eq2.2], it is trivial that $\mu \geq 0$. We prove $\mu>0$ by contradiction. We assume that $\mu=0$. Taking the minimizing sequence $\{\psi_n\}\subset X_{\omega,c}$, i.e. $\til{S}(\psi_n) \to \mu=0$ and $\til{K}(\psi_n) =0$, we have $\norm[ \partial_x \psi_n]_{L^2}^2 \to 0$ and $\norm[\psi_n]_{L^{2\sigma+2}}^{2\sigma+2} \to 0$ by \eqref[eq2.2] and \eqref[eq2.3]. 
Then, by using \eqref[eq2.1] as $p=(\sigma+2)/2$, we get $\norm[\psi_n]_{L^\infty}\to0$ as $n \to \infty$. 
By using 
\begin{align*}
-N(\psi)=-\norm[\partial_x \psi]_{L^2}^2 - \frac{1}{4} \norm[\psi]_{L^{4\sigma+2}}^{4\sigma+2} +\norm[\partial_x \psi + \frac{1}{2}i|\psi|^{2\sigma} \psi]_{L^2}^2,
\end{align*}
we obtain
\begin{align*}
\til{K}(\psi_n)
&=\frac{2\alpha-\beta}{2}\norm[\partial_x \psi_n]_{L^2}^2 
+\frac{\{(2\sigma+2)\alpha+\beta\} c}{2(2\sigma+2)} \norm[\psi_n]_{L^{2\sigma+2}}^{2\sigma+2} -\alpha N(\psi_n)
\\
&=-\frac{\beta}{2}\norm[\partial_x \psi_n]_{L^2}^2 
+\frac{\{(2\sigma+2)\alpha+\beta\} c}{2(2\sigma+2)} \norm[\psi_n]_{L^{2\sigma+2}}^{2\sigma+2} -\frac{\alpha}{4}\norm[\psi_n]_{L^{4\sigma+2}}^{4\sigma+2} 
\\
& \quad +\alpha \norm[\partial_x \psi_n + \frac{1}{2}i|\psi_n|^{2\sigma} \psi_n]_{L^2}^2
\\
&\geq  \frac{\{(2\sigma+2)\alpha+\beta\} c}{2(2\sigma+2)} \norm[\psi_n]_{L^{2\sigma+2}}^{2\sigma+2} -\frac{\alpha}{4}\norm[\psi_n]_{L^{4\sigma+2}}^{4\sigma+2} 
\\
&\geq  \frac{\{(2\sigma+2)\alpha+\beta\} c}{2(2\sigma+2)} \norm[\psi_n]_{L^{2\sigma+2}}^{2\sigma+2} -\frac{\alpha}{4} \norm[\psi_n]_{L^{2\sigma+2}}^{2\sigma+2}\norm[\psi_n]_{L^\infty}^{2\sigma}
\\
&\geq \l( \frac{\{(2\sigma+2)\alpha+\beta\} c}{2(2\sigma+2)}  -\frac{\alpha}{4} \norm[\psi_n]_{L^\infty}^{2\sigma}\r)\norm[\psi_n]_{L^{2\sigma+2}}^{2\sigma+2}
\\
&>0, 
\end{align*}
for large $n\in \N$ since $\norm[\psi_n]_{L^\infty}\to0$ as $n \to \infty$. However, this contradicts $\til{K}(\psi_n)=0$ for all $n\in \N$. 
\end{proof}

\begin{proof}[Proof of {\propref[prop2.10]}] 
We take $\{\psi_n\} \subset X_{\omega,c}$ such that  $\til{S}_{\omega,c}(\psi_n) \to \mu_{\omega,c}^{\alpha,\beta}$ and $ \til{K}_{\omega,c}^{\alpha,\beta}(\psi_n) \to 0$. Then, $\{\psi_n\} $ is a bounded sequence in $X_{\omega,c}$ by \eqref[eq2.2]. 
\\
{\bf Step 1.} 
We prove $ \limsup_{n\to\infty}\norm[\psi_n]_{L^{4\sigma+2}}>0$ by contradiction. We suppose that $\limsup_{n\to\infty}\norm[\psi_n]_{L^{4\sigma+2}}=0$. Since we have
\begin{align*}
0\gets \til{K}(\psi_n)
&\geq -\frac{\beta}{2}\norm[\partial_x \psi_n]_{L^2}^2 
+\frac{\{(2\sigma+2)\alpha+\beta\} c}{2(2\sigma+2)} \norm[\psi_n]_{L^{2\sigma+2}}^{2\sigma+2} -\frac{\alpha}{4}\norm[\psi_n]_{L^{4\sigma+2}}^{4\sigma+2}, 
\end{align*}
we obtain $\norm[\partial_x \psi_n]_{L^2}^2 \to0$ and $\norm[\psi_n]_{L^{2\sigma+2}}^{2\sigma+2} \to 0 $ as $n \to \infty$. By \eqref[eq2.2], we get $\til{S}(\psi_n) \to 0$. This contradicts that $\mu>0$. 
\\
{\bf Step 2.} 
Since $\{\psi_n\} $ is bounded in $X_{\omega,c}= \dot{H}^1(\R)\cap L^{2\sigma+2}(\R)$ and $ \limsup_{n\to\infty}\norm[\psi_n]_{L^{4\sigma+2}}>0$, by applying \lemref[lem2.3] as $f_n=\psi_n$, $d=1$, and $p=2\sigma+2$, there exists $\{y_n\}$ and $v \in X_{\omega,c} \setminus \{0\}$ such that $\{\psi_n(\cdot-y_n)\}$ (we denote this by $v_n$) has a subsequence that converges to $v$ weakly in $X_{\omega,c}$. 
\\
{\bf Step 3.} 
We show 
\begin{equation}
\label{eq2.5}
\til{K}(v_n) -\til{K}(v-v_n) -\til{K}(v) \to 0 \text{ as } n \to \infty.
\end{equation}
We note that 
\begin{align} 
\label{eq2.6}
\til{K}(\psi)
=-\frac{\beta}{2}\norm[\partial_x \psi]_{L^2}^2 
&+\frac{\{(2\sigma+2)\alpha+\beta\} c}{2(2\sigma+2)} \norm[\psi]_{L^{2\sigma+2}}^{2\sigma+2}
\\ \notag
& -\frac{\alpha}{4}\norm[\psi]_{L^{4\sigma+2}}^{4\sigma+2}+\alpha \norm[\partial_x \psi + \frac{i}{2}|\psi|^{2\sigma} \psi]_{L^2}^2, 
\end{align}
for any $\psi \in X_{\omega,c}$. Since $v_n$ converges to $v$ weakly in $X_{\omega,c}$, we have $v_n \to v$ a.e. in $\R$. Therefore, by \lemref[lem2.4], we have $\norm[v_n]_{L^p}^p-\norm[v_n-v]_{L^p}^p-\norm[v]_{L^p}^p \to 0$ for $2\sigma+2 \leq p < \infty$. Moreover, setting 
\[w_n:=\partial_x v_n +\frac{i}{2} |v_n|^{2\sigma}v_n \text{ and } w=\partial_x v +\frac{i}{2} |v|^{2\sigma}v, \]
$w_n$ converges to $w$ weakly in $L^2(\R)$. Indeed, it is obvious that $\partial_x v_n \wto \partial_x v$ in $L^2(\R)$ and we have, for any $f \in C_{0}^{\infty}(\R)$,  
\begin{align*} 
\int_{\R} f(x) &(|v_n(x)|^{2\sigma}v_n(x)- |v(x)|^{2\sigma}v(x)) dx
\\
&\cleq \int_{\supp f} |f(x)| (|v_n(x)|^{2\sigma}+|v(x)|^{2\sigma})|v_n(x)- v(x)|dx
\\
& \cleq \int_{\supp f} |v_n(x)- v(x)|dx \to 0, 
\end{align*}
where we use the H\"{o}lder inequality, the fact that $\{v_n\}$ is bounded in $L^\infty(\R)$, the compactness of the embedding $\dot{H}^1(\Omega)\cap L^{2\sigma+2}(\Omega) \hookrightarrow H^1(\Omega)\hookrightarrow  L^{p}(\Omega)$ for a bounded domain $\Omega \subset \R$ and $1 \leq  p \leq \infty$. Thus,  $w_n$ converges to $w$ weakly in $L^2(\R)$. Therefore, by \eqref[eq2.6], we get \eqref[eq2.5].
\\
{\bf Step 4.} 
We prove $\alpha(2\sigma+2) \mu< \til{J}(\psi)$ if $\til{K}(\psi)<0$. By the definition of $\mu$, we have
\begin{equation}
\label{eq2.7}
\mu_{\omega,c}^{\alpha,\beta} 
=\frac{1}{\alpha(2\sigma+2)} \inf\{ \til{J}_{\omega,c}^{\alpha,\beta}(\psi):\psi \in X_{\omega,c}\setminus\{0\}, \til{K}_{\omega,c}^{\alpha,\beta}(\psi)=0 \}.
\end{equation}
If $\psi\in X_{\omega,c}$ satisfies $\til{K}(\psi)<0$, then there exists $\lambda_0 \in (0,1)$ such that $\til{K}(\lambda_0 \psi)=0$ since $\til{K}(\lambda \psi)>0$ for small $\lambda \in (0,1)$. Therefore, we have $\alpha(2\sigma+2) \mu \leq \til{J}(\lambda_0 \psi) < \til{J}(\psi)$.
\\
{\bf Step 5.} 
We prove $\til{K}(v) \leq 0$ by contradiction. We suppose that $\til{K}(v) > 0$. Since $\til{K}(v_n) \to 0$ and \eqref[eq2.5] hold, we have
\[ \til{K}(v-v_n) \to -\til{K}(v) <0.  \]
This implies that $\til{K}(v-v_n)<0$ for large $n\in\N$. Therefore, by Step 4, we get $\alpha(2\sigma+2)\mu<\til{J}(v-v_n)$ for large $n \in \N$. By the same argument as in Step 3, we get
\[ \til{J}(v_n) -\til{J}(v-v_n) -\til{J}(v) \to 0 \text{ as } n \to \infty.\]
Therefore, we get $\til{J}(v) = \lim_{n \to \infty} (\til{J}(v_n) -\til{J}(v-v_n)) \leq 0$ since we have $\til{J}(v_n) \to \alpha(2\sigma+2)\mu$ by the definition of $\til{J}$ and $\til{K}(v_n) \to0$. By Step 2, we have $v\neq0$ and then $\til{J}(v)>0$. This is a contradiction. 
\\
{\bf Step 6.} 
We prove that $v$ belongs to $\til{\scM}$. 
By \eqref[eq2.7] and the weakly lower semicontinuity of $\til{J}$, we obtain
\[ \alpha(2\sigma+2) \mu \leq \til{J}(v) \leq \liminf_{n \to \infty} \til{J}(v_n) =\alpha(2\sigma+2)\mu.  \]
Thus, $\til{J}(v)= \alpha(2\sigma+2) \mu$ and $v_n$ converges to $v$ strongly in $X_{\omega,c}$. Therefore, we get $\til{S}(v)=\mu$ and $\til{K}(v)=0$ by Step 4 and 5. 
\end{proof}

Therefore, we obtain \propref[prop2.1] when $\omega=c^2/4$ and $c>0$. 

The case of $\omega>c^2/4$ is much easier.  Indeed, we can obtain \propref[prop2.1] by the same argument as in the case $\omega=c^2/4$ and $c>0$ by using $L^2(\R)$ instead of $L^{2\sigma+2}(\R)$. See also Colin and Ohta \cite{CO06}, where they obtained the statement only for the Nehari functional $K_{\omega,c}^{1,0}$. Thus, we omit the proof.

\section{Global Well-Posedness}
\label{sec3}

In this section, we show \thmref[thm1.2].  
\begin{proof}[Proof of {\thmref[thm1.2]}]
Let $u_0$ belong to $\scK_{\omega,c}^{\alpha,\beta,+}$. 
First, we consider the case that $K_{\omega,c}^{\alpha,\beta}(u_0)= 0$. Then, $u_0=0$ or $u_0=e^{i\theta_0} \phi_{\omega,c}(\cdot-x_0)$ by \thmref[thm1.1]. By the uniqueness of solution to \eqref[gDNLS], we have $u(t)=0$ or $u(t)=e^{i\theta_0} e^{i\omega t} \phi_{\omega,c}(x-ct-x_0)$, respectively. This implies that $K_{\omega,c}^{\alpha,\beta}(u(t))= 0$ for all time. This means that $u(t) \in \scK_{\omega,c}^{\alpha,\beta,+}$ for all time. Next, we consider the case that $K_{\omega,c}^{\alpha,\beta}(u_0)> 0$. We suppose that there exists a time $t$ such that $K_{\omega,c}^{\alpha,\beta}(u(t)) \leq 0$. Then, there exists $t_*$ such that $K_{\omega,c}^{\alpha,\beta}(u(t_*)) = 0$ by the continuity of the flow. By the above argument, $K_{\omega,c}^{\alpha,\beta}(u(t))= 0$ for all time. This is a contradiction. Thus, $u(t)$ belongs to $\scK_{\omega,c}^{\alpha,\beta,+}$ for all time. 
When $u_0$ belongs to $\scK_{\omega,c}^{\alpha,\beta,-}$, the same argument implies that $u(t)$ belongs to $\scK_{\omega,c}^{\alpha,\beta,-}$ for all time. 
Next, we prove that the solution is global if $u_0 \in \scK_{\omega,c}^{\alpha,\beta,+}$. Then, since we have 
\begin{align}
\label{eq3.1}
\alpha (2\sigma+2) S_{\omega,c}(\varphi)
&=K_{\omega,c}^{\alpha,\beta}(\varphi)
+\frac{2\sigma\alpha+\beta}{2}\norm[\partial_x \varphi-\frac{c}{2}i \varphi]_{L^2}^2 
\\ \notag
&\quad +\l( \omega-\frac{c^2}{4}\r)\frac{2\sigma\alpha-\beta}{2}\norm[\varphi]_{L^2}^2 -\frac{\beta c}{2(2\sigma+2)}\norm[\varphi]_{L^{2\sigma+2}}^{2\sigma+2}
\end{align}
and $K_{\omega,c}^{\alpha,\beta}(u(t))> 0$ for all time $t$, we have that $\norm[\partial_x u(t)-\frac{c}{2}i u(t)]_{L^2}^2$ is uniformly bounded. 
Therefore, we  have
\[ \norm[\partial_xu(t)]_{L^2} \leq \norm[\partial_x u(t)-\frac{c}{2}i u(t)]_{L^2} + \frac{|c|}{2}\norm[u(t)]_{L^2} < C+ \frac{|c|}{2}\norm[u_0]_{L^2}, \]
for some positive constant $C$ independent of $t$. 
This boundedness and the conservation law of the $L^2$-norm imply that $u$ is global in time. 
\end{proof}

We give proofs of Corollaries \ref{cor1.3}, \ref{cor1.4}, and  \ref{cor1.5}.
Direct calculations imply the following lemma (see \cite{CO06} for the detail).

\begin{lemma}
\label{lem3.1}
Let $\sigma=1$ and $(\omega,c)$ satisfy \eqref[eq1.5]. Then, we have the following relations: 
\begin{align*}
M(\phi_{\omega,c})&=8\tan^{-1}\sqrt{\frac{2\sqrt{\omega}+c}{2\sqrt{\omega}-c}},
\\
P(\phi_{\omega,c})&=2\sqrt{4\omega-c^2},
\\
E(\phi_{\omega,c})&=-\frac{c}{2}\sqrt{4\omega-c^2}.
\end{align*}
In particular, we have 
\[ S_{\omega,c}(\phi_{\omega,c})= 4\omega \tan^{-1}\sqrt{\frac{2\sqrt{\omega}+c}{2\sqrt{\omega}-c}} + \frac{c}{2}\sqrt{4\omega-c^2}. \]
\end{lemma} 

\begin{remark}
When $\sigma=1$, we have $M(\phi_{c^2/4,c})=4\pi$, $P(\phi_{c^2/4,c})=0$, and $E(\phi_{c^2/4,c})=0$ for all $c>0$ by \lemref[lem3.1]. On the other hand, if $M(\phi)=4\pi$, $P(\phi)=0$, and $E(\phi)\le 0$, then $\phi(x)=e^{i\theta_0}\phi_{c_0^2/4,c_0}(x-x_0)$ for some $\theta_0\in\R$, $x_0\in\R$, and $c_0 >0$. Indeed, $M(\phi)=4\pi$, $P(\phi)=0$, and $E(\phi)\le 0$ imply that
\begin{align*}
K_{c^2/4,c}^{\alpha,\beta}(\phi)
&\le - \frac{2\alpha +\beta}{2}\left\Vert \partial_x \phi \right\Vert_{L^2}^2 + \frac{2\alpha-\beta}{2} c^2 \pi +\frac{\beta c}{8} \left\Vert \phi \right\Vert_{L^4}^{4}.
\end{align*}
Since $K_{c^2/4,c}^{\alpha,\beta}(\phi)<0$ for small $c>0$ and $K_{c^2/4,c}^{\alpha,\beta}(\phi) \to +\infty$ as $c \to \infty$, there exists $c_0>0$ such that $K_{c_0^2/4,c_0}^{\alpha,\beta}(\phi)=0$. Therefore, \thmref[thm1.1] implies that $\phi(x)=e^{i\theta_0}\phi_{c_0^2/4,c_0}(x-x_0)$. Note that this means that there is no function satisfying $M(\phi)=4\pi$, $P(\phi)=0$, and $E(\phi)<0$. 
\end{remark}

First, we prove \corref[cor1.3].

\begin{proof}[Proof of {\corref[cor1.3]}]
Let $u_0$ satisfy $\norm[u_0]_{L^2}^2<4\pi$. The statement is trivial if $u_0=0$. We assume that $u_0\neq0$. Since $\norm[u_0]_{L^2}^2<4\pi$, we have
\[ S_{c^2/4,c}(u_0)=E(u_0)+\frac{c^2}{8}\norm[u_0]_{L^2}^2 +\frac{c}{2}P(u_0)< c^2 \pi/2,  \]
for sufficiently large $c>0$. Moreover, since $\norm[u_0]_{L^2}^2\neq 0$, we have
\begin{align}
\label{eq3.2}
K_{c^2/4,c}^{\alpha,\beta}(u_0) 
&= \frac{2\alpha-\beta}{2}\norm[\partial_x u_0]_{L^2}^2 + \frac{2\alpha-\beta}{2} \frac{c^2}{4} \norm[u_0]_{L^2}^2 +\frac{2\alpha-\beta}{2}c P(u_0) 
\\ \notag
& \quad+\frac{\beta c}{8} \norm[u_0]_{L^{4}}^{4} -\alpha N(u_0)
\\ \notag
& \to \infty \text{ as } c \to \infty,
\end{align}
for any $(\alpha,\beta)$ satisfying \eqref[eq1.7]. 
Thus, $K_{c^2/4,c}^{\alpha,\beta}(u_0)>0$ for large $c>0$. Thus, there exists $c>0$ such that $K_{c^2/4,c}^{\alpha,\beta}(u_0)>0$ and $S_{c^2/4,c}(u_0)< c^2 \pi/2$, where we note that $\mu_{c^2/4,c}=c^2 \pi/2$ by \lemref[lem3.1] when $\sigma=1$. By \thmref[thm1.2], the solution $u$ is global. 
\end{proof}

Secondly, we give a proof of \corref[cor1.4] by \thmref[thm1.2]. 

\begin{proof}[Proof of {\corref[cor1.4]}]
Let $u_0$ satisfy $\norm[u_0]_{L^2}^2=4\pi$ and $P(u_0)<0$. We recall that $\mu_{c^2/4,c}=c^2 \pi/2$ by \lemref[lem3.1] when $\sigma=1$.
Since $P(u_0)<0$, we have, for large $c>0$, 
\[ S_{c^2/4,c}(u_0)=E(u_0) + \frac{c^2}{2} \pi +\frac{c}{2}P(u_0) \leq  \mu_{c^2/4,c}. \]
On the other hand, by $2\alpha-\beta>0$ and $\norm[u_0]_{L^2}^2\neq 0$, we obtain \eqref[eq3.2]. 
Thus, $K_{c^2/4,c}^{\alpha,\beta}(u_0)>0$ for large $c>0$. This means that the assumption in \thmref[thm1.2] holds for sufficiently large $c$. This implies that $u$ is global.
\end{proof}

We give another proof. This is due to Guo and Wu \cite{GW17}.
\begin{proof}[Another proof of {\corref[cor1.4]}]
We have 
\begin{align*}
P(u) \geq \frac{1}{4} \norm[u]_{L^4}^4 \l( 1-\frac{\norm[u]_{L^2}}{2\sqrt{\pi}}\r) - \frac{8\sqrt{\pi}E(u)\norm[u]_{L^2}}{\norm[u]_{L^4}^4},
\end{align*}
applying the gauge transformation $u=w\exp(-\frac{i}{4}\int_{-\infty}^{x} |w(y)|^2dy)$ to \eqref[eq1.15]. 
See \cite[Lemma 2.1]{GW17} for the proof of \eqref[eq1.15]. When $\norm[u_0]_{L^2}^2=4\pi$ and $P(u_0)<0$, we get 
\begin{align}
\label{eq3.3}
\norm[u(t)]_{L^4}^4  \leq  \frac{8\sqrt{\pi}E(u_0)\norm[u_0]_{L^2}}{|P(u_0)|}.
\end{align}
Therefore, by the H\"{o}lder inequality, the Gagliardo--Nirenberg inequality, and the Young inequality, we have
\begin{align*}
	\norm[\partial_x u(t)]_{L^2}^2 
	&= 2 E(u_0) +\frac{1}{2}\re \int_{\R} i |u(t,x)|^2 \overline{u(t,x)} \partial_x u(t,x) dx
	\\
	& \leq 2E(u_0) + \frac{1}{2} \norm[u(t)]_{L^6}^3 \norm[ \partial_x u(t)]_{L^2}
	\\
	& \leq 2E(u_0) + C \norm[u(t)]_{L^4}^{8/3} \norm[ \partial_x u(t)]_{L^2}^{4/3}
	\\
	& \leq 2E(u_0) + C \norm[u(t)]_{L^4}^{8} + \frac{1}{2} \norm[ \partial_x u(t)]_{L^2}^{2}.
\end{align*}
This inequality and \eqref[eq3.3] give a priori estimate and thus the solution is global. 
\end{proof}

At last, we prove \corref[cor1.5]. 

\begin{proof}[Proof of {\corref[cor1.5]}]
Let $\sigma \geq 1$. 
Since $u_{0,c}=e^{\frac{c}{2}ix}\psi$, we have
\begin{align*} 
S_{c^2/4,c}(u_{0,c})
&=\til{S}_{c^2/4,c}(\psi)
\\
&=\frac{1}{2}\norm[\partial_x \psi]_{L^2}^2  +\frac{c}{2(2\sigma+2)}\norm[\psi]_{L^{2\sigma+2}}^{2\sigma+2} -\frac{1}{2\sigma+2}N(\psi)
\\
& \leq c^{1+\frac{1}{\sigma}}S_{1/4,1}(\phi_{1/4,1})= S_{c^2/4,c}(\phi_{c^2/4,c}),
\\
K_{c^2/4,c}^{\alpha,\beta}(u_{0,c}) 
&=\til{K}_{c^2/4,c}^{\alpha,\beta}(\psi)
\\
&=\frac{2\alpha-\beta}{2}\norm[\partial_x \psi]_{L^2}^2 
+\frac{\{(2\sigma+2)\alpha+\beta\} c}{2(2\sigma+2)} \norm[\psi]_{L^{2\sigma+2}}^{2\sigma+2} -\alpha N(\psi)
\\
&\geq 0, 
\end{align*}
for large $c>0$. By \thmref[thm1.2], therefore, the solution $u_c$ with the initial data $u_{0,c}$ is global for large $c>0$. 
\end{proof}

\appendix
\section{Uniqueness and Non-existence}
\label{appA}

We prove the uniqueness of the massless elliptic equation.

\begin{proposition}
Let $1<p<q<\infty$, $a>0$, and $b>0$. Assume there exists a non-trivial solution in $\dot{H}^1(\R)\cap L^{p+1}(\R)$ of the equation
\begin{align}
\label{eqA.1}
-\varphi'' +a|\varphi|^{p-1}\varphi - b|\varphi|^{q-1}\varphi =0
\end{align} 
 in the distribution sense. Then there exist $\theta_0 \in [0,2\pi ), x_0\in \R$ such that $\varphi = e^{i\theta_0} \psi (\cdot-x_0)$, where $\psi$ is the unique positive, even, and decreasing function which satisfies \eqref[eqA.1].
\end{proposition}

\begin{proof}
Since $a|\varphi|^{p-1}\varphi - b|\varphi|^{q-1}\varphi$  belongs to $L^2(\R)$, we obtain $\varphi \in \dot{H}^2(\R)$. A bootstrap argument gives us that $\varphi \in \dot{H}^3(\R)$. By the Sobolev embedding, $\varphi \in C^2(\R)$ and $\varphi$ satisfies the equation in the classical sense. Multiplying the equation by $\varphi'$ and integrating on $(-\infty ,x)$, we obtain
\begin{equation} 
\label{eqA.2}
-\frac{1}{2}|\varphi'( x)|^2 +\frac{a}{p+1}|\varphi(x)|^{p+1} -\frac{b}{q+1}|\varphi(x)|^{q+1}=0.
\end{equation}
We write $\varphi = \rho e^{i\theta}$, where $\rho >0$ and $\rho , \theta \in C^2(\mathbb{R})$. It is easily seen that $\theta \equiv \theta_0$ for some $\theta_0 \in [0, 2\pi )$. Since $\rho \in L^{p+1}(\R )$, there must exist $x_0 \in \R$ such that $\rho'(x_0)=0$. By \eqref[eqA.2], $\rho(x_0)=c$, where $c^{q-p}=\frac{a(q+1)}{b(p+1)}$. Let $\psi$ be the real-valued solution of \eqref[eqA.1] such that $\psi(0)=c$ and $\psi '(0)=0$. 
Using the uniqueness of the ordinary differential equation, we can deduce that $\varphi = e^{i\theta_0} \psi (\cdot-x_0)$. 
\end{proof}

We prove the non-existence of a non-trivial solution to the nonlinear elliptic equation \eqref[elliptic] in the case $\omega<c^2/4$, or $\omega=c^2/4$ and $c\leq 0$. This is covered by \cite[Theorem 5]{BL83}.  See \cite[Theorem 5]{BL83} for the necessary and sufficient condition for the existence of non-trivial solutions to more general second order ordinary differential equations.

\begin{proposition}
\label{propA.2}
Let $1<p,q<\infty$. If $\varphi \in H^1(\R)$ satisfies
\[ -\varphi'' +\omega \varphi +a|\varphi|^{p-1}\varphi-b|\varphi|^{q-1}\varphi=0 \text{ in the distribution sense,} \]
where $a>0$, $b>0$, and $\omega<0$, then we have $\varphi=0$. 
\end{proposition}

\begin{proof}
By a usual bootstrap argument (see Section 8 in \cite{Caz03}), we have $\varphi \in H^3(\R)$. We get $\varphi \in C^2(\R)$ by the Sobolev embedding. Therefore, $\varphi '(x) \to 0$ and $\varphi(x)\to0$ as $|x|\to\infty$. Multiplying the equation by $\varphi'$ and integrating on $(-\infty ,x)$, we obtain
\begin{equation} 
\label{eqA.3}
-\frac{1}{2}|\varphi'(x)|^2 +\frac{\omega}{2}|\varphi(x)|^2 +\frac{a}{p+1}|\varphi(x)|^{p+1} -\frac{b}{q+1}|\varphi(x)|^{q+1}=0.
\end{equation}
Since $\varphi(x)\to0$ as $|x|\to\infty$, we get 
\[  \frac{\omega}{2}|\varphi(x)|^2 +\frac{a}{p+1}|\varphi(x)|^{p+1} -\frac{b}{q+1}|\varphi(x)|^{q+1}< 0, \text{ for some }x \]
or 
\[  |\varphi(x)|= 0, \text{ for some }x. \]
In the former case, we obtain $|\varphi'(x)|< 0$ by \eqref[eqA.3]. This is a contradiction. In the latter case, we obtain $|\varphi'(x)|=0$ by \eqref[eqA.3]. By the uniqueness of the ordinary differential equation, we get $\varphi=0$.
\end{proof}

By the same argument as in the proof of \propref[propA.2], we obtain the non-existence of a non-trivial solution to the nonlinear elliptic equation \eqref[elliptic] when $\omega=c^2/4$ and $c\leq 0$ as follows.

\begin{proposition}
Let $1<p,q<\infty$. If $\varphi \in \dot{H}^1(\R)\cap L^{p+1}(\R)$ satisfies
\[ -\varphi'' -a|\varphi|^{p-1}\varphi-b|\varphi|^{q-1}\varphi=0 \text{ in the distribution sense,} \]
where $a \geq 0$ and $b>0$, then we have $\varphi=0$. 
\end{proposition}

\section*{Acknowledgement}
The authors would like to express deep appreciation to Professor Kenji Nakanishi for constant encouragement, Professor Masahito Ohta for many useful suggestions, and Professor Tohru Ozawa for advice on notations. The third author is supported by Grant-in-Aid for JSPS Research Fellow 15J02570. The authors also would like to thank Guixiang Xu for introducing their works, Zihua Guo for a suggestion about Corollary \ref{cor1.4},  and the anonymous referee for his variable comments.


\end{document}